\documentclass[reqno,11pt]{amsart}
\usepackage{latexsym}
\usepackage{amsmath,amsfonts,amssymb,epsfig}
\usepackage{amsthm, amscd}
\usepackage{xcolor}
\usepackage{setspace}
\usepackage{overpic}
\usepackage{eucal}
\usepackage{mathrsfs}
\usepackage{wrapfig}
\newtheorem*{rep@theorem}{\rep@title}
\newcommand{\newreptheorem}[2]{%
	\newenvironment{rep#1}[1]{%
		\def\rep@title{#2 \ref{##1}}%
		\begin{rep@theorem}}%
		{\end{rep@theorem}}}

\newtheorem{thm}{Theorem}
\newreptheorem{thm}{Theorem}

\newtheorem{cor}[thm]{Corollary}
\newtheorem{lem}[thm]{Lemma}

\newtheorem{rem}[thm]{Remark}

\newtheorem{prob}{Question}

\newcommand{\no}{\noindent}

\def\crn{\operatorname{\mathsf{cr}}} 
\def\Crn{\operatorname{\mathsf{Cr}}} 
\def\acr{\operatorname{\mathsf{acr}}}

\def\aov{\operatorname{\mathsf{aov}}}

\def\PCrn{\operatorname{\mathsf{PCr}}}

\def\Ac{\operatorname{\mathsf{Ac}}}

\def\rop{\operatorname{\mathsf{rop}}} 
\def\Rop{\operatorname{\mathsf{Rop}}} 
\def\lk{\operatorname{\mathsf{Lk}}} 
\def\thk{\operatorname{\tau}}

\def\Ln{\ell}

\def\R{\mathbb{R}}

\def\Z{\mathbb{Z}}

\def\Cnf{\text{\rm Conf}}
\def\sign{\text{\rm sign}}

\numberwithin{equation}{section} 
\newcommand{\vcenteredinclude}[1]{\begingroup
\setbox0=\hbox{\includegraphics[scale=.5]{#1}}%
\parbox{\wd0}{\box0}\endgroup}

\newcommand{\vvcenteredinclude}[2]{\begingroup
\setbox0=\hbox{\includegraphics[scale=#1]{#2}}%
\parbox{\wd0}{\box0}\endgroup}  

\title[Ropelength and finite type invariants]{Ropelength, crossing number and finite type\\ invariants of links}

\author{R. Komendarczyk}
\address{%Department of Mathematics,
	Tulane University,
	New Orleans, Louisiana 70118 } 
\email{rako@tulane.edu}

\thanks{Supported by NSF DMS 1043009 and DARPA YFA N66001-11-1-4132 during the years 2011-2015} 

\author{A. Michaelides}

\address{%Department of Mathematics,
	University of South Alabama,
	Mobile, AL 36688 } 
\email{amichaelides@southalabama.edu}

\begin{document}
%\onehalfspacing

\begin{abstract}
	{\em Ropelength} and {\em embedding thickness} are related measures of geometric complexity of classical knots and links in Euclidean space. In their recent work, Freedman and Krushkal posed a question regarding lower bounds for embedding thickness of $n$-component links in terms of the Milnor linking numbers. The main goal of the current paper is to provide such estimates, and thus generalizing the known linking number bound. In the process, we collect several facts about finite type invariants and ropelength/crossing number of knots. We give examples of families of knots, where such estimates behave better than the well known knot--genus estimate.
 \end{abstract}

\maketitle

%%%%%%%%%%%%%%%%%%%%%%%%%%%%%%%%%%%%%%%%%%%%%%%%%%%%%%%%%%
\section{Introduction}\label{S:intro}

 Given an $n$--component link (we assume class $C^1$ embeddings) in $3$--space 
 \begin{equation}\label{eq:n-link}
  L:S^1\sqcup\ldots \sqcup S^1\longrightarrow \R^3,\qquad L=(L_1, L_2,\ldots, L_n),\quad L_i=L|_{\text{the $i$'th circle}},
 \end{equation}  
its {\em ropelength} $\rop(L)$ is the ratio $\rop(L)=\frac{\Ln(L)}{r(L)}$ of length $\Ln(L)$, which is a sum of lengths of individual components of $L$,  to  {\em reach} or {\em thickness}: $r(L)$, i.e. the largest radius of the tube embedded as a normal neighborhood of $L$. The {\em ropelength within the isotopy class} $[L]$ of $L$ is defined as    
 \begin{equation}\label{eq:Rop-def}
  \Rop(L)=\inf_{L'\in [L]} \rop(L'),\qquad \rop(L')=\frac{\Ln(L')}{r(L')},
 \end{equation}
 (in \cite{Cantarella-Kusner-Sullivan:2002} it is shown that the infimum is achieved within $[L]$ and the minimizer is of class $C^{1,1}$).
 A related measure of complexity, called {\em embedding thickness} was introduced recently in \cite{Freedman-Krushkal:2014}, in the general context of embeddings' complexity. For links, the embedding thickness $\tau(L)$ of $L$ is given by a value of its reach $r(L)$ assuming that $L$ is a subset of the unit ball $B_1$ in $\R^3$ (note that any embedding can be scaled and translated to fit in $B_1$). Again, the embedding thickness of the isotopy class $[L]$ is given by
 %%%%%%%%%%%%%%%%%%%%%%%%%%%%%%%%%%%%%%%%%%%%%%%%%%%%%%%%%%%%
 \begin{equation}\label{eq:thck-def}
 \mathcal{T}(L)=\sup_{L'\in [L]} \tau(L').
 \end{equation}   
 For a link $L\subset B_1$, the volume of the embedded tube of radius $\tau(L)$ is $\pi \Ln(L)\tau(L)^2$, \cite{Gray:2004} and the tube is contained in the ball of radius $r=2$, yielding
 \begin{equation}\label{eq:t(L)-rop(L)}
  \rop(L)=\frac{\pi\Ln(L)\tau(L)^2}{\pi\tau(L)^3}\leq \frac{\frac 43 \pi 2^3}{\pi\tau(L)^3},\quad \Rightarrow\quad \tau(L)\leq \Bigl( \frac{11}{\rop(L)}\Bigr)^{\frac 13}.
 \end{equation}
 In turn a lower bound for $\rop(L)$ gives an upper bound for $\tau(L)$ and vice versa.
For other measures of complexity of embeddings such as distortion or Gromov-Guth thickness see e.g. \cite{Gromov:1983}, \cite{Gromov-Guth:2012}.

 Bounds for the ropelength of knots, and in particular the lower bounds,  have been studied by many researchers, we only list a small fraction of these works here
\cite{Buck-Simon:1999,Buck-Simon:2007,Cantarella-Kusner-Sullivan:2002,Diao-Ernst-Janse-van-Rensburg:1999,Diao-Ernst:2007,Diao:2003,Ernst-Por:2012,Litherland-Simon-Durumeric-Rawdon:1999,Rawdon:1998, Ricca-Maggioni:2014, Maggioni-Ricca:2009, Ricca-Moffatt:1992}. Many of the results are applicable directly to links, but the case of links is treated in more detail by Cantarella, Kusner and Sullivan \cite{Cantarella-Kusner-Sullivan:2002} and in the earlier work of Diao,  Ernst, and Janse Van Rensburg
\cite{Diao-Ernst-Janse-Van-Rensburg:2002} concerning the estimates in terms of the pairwise linking number. In \cite{Cantarella-Kusner-Sullivan:2002}, the authors introduce a cone surface technique and show the following estimate, for a link $L$ (defined as in \eqref{eq:n-link}) and a given component $L_i$ \cite[Theorem 11]{Cantarella-Kusner-Sullivan:2002}:
%%%%%%%%%%%%%%%%%%%%%%%%%%%%%%%%%%%%%%%%%%%%%%%%%%%%%%%%%%%
\begin{equation}\label{eq:len(L_i)-lk}
\rop(L_i)\geq 2\pi+2\pi\sqrt{\lk(L_i,L)},
\end{equation}
where $\lk(L_i,L)$ is the maximal total linking number between $L_i$ and the other components of $L$. A stronger estimate was obtained in  \cite{Cantarella-Kusner-Sullivan:2002} by combining the Freedman and He \cite{Freedman-He:1991} asymptotic crossing number bound for energy of divergence free fields and the cone surface technique as follows
%%%%%%%%%%%%%%%%%%%%%%%%%%%%%%%%%%%%%%%%%%%%%%%%%%%%%%%%%%%
\begin{equation}\label{eq:len(L_i)-Ac-gen}
\rop(L_i)\geq 2\pi+2\pi\sqrt{\Ac(L_i,L)},\qquad \rop(L_i)\geq 2\pi+2\pi\sqrt{2g(L_i,L)-1},
\end{equation}
where $\Ac(L_i,L)$ is the {\em asymptotic crossing number} (c.f. \cite{Freedman-He:1991}) and the second inequality is a consequence of the estimate $\Ac(L_i,L)\geq 2g(L_i,L)-1$, where $g(L_i,L)$ is a minimal genus among surfaces embedded in $\R^3\setminus\{L_1\cup\ldots\cup\widehat{L_i}\cup\ldots\cup L_n\}$, \cite[p. 220]{Freedman-He:1991} (in fact, Estimate \eqref{eq:len(L_i)-Ac-gen} subsumes Estimate \eqref{eq:len(L_i)-lk} since $\Ac(L_i,L)\geq \lk(L_i,L)$). 
 A relation between $\Ac(L_i,L)$ and the higher linking numbers of Milnor, \cite{Milnor:1954, Milnor:1957} is unknown and appears difficult.
The following question, concerning the embedding thickness, is stated in \cite[p. 1424]{Freedman-Krushkal:2014}:
%%%%%%%%%%%%%%%%%%%%%%%%%%%%%%%%%%%%%%%%%%%%%%%%%%%%%%%%%%%
\begin{prob}\label{q:F-K}
	Let $L$ be an $n$-component link which is Brunnian (i.e. almost trivial in the sense of Milnor \cite{Milnor:1954}). Let $M$ be the maximum value among Milnor's $\bar{\mu}$-invariants with distinct indices i.e. among $|\bar{\mu}_{\mathtt{I};j}(L)|$. Is there a bound 
	\begin{equation}\label{eq:thk-FK}
	\thk(L) \leq c_n M^{-\frac{1}{n}},
	\end{equation}
	for some constant $c_n > 0$, independent of the link $L$? Is there a bound on the crossing number $\Crn(L)$ in terms of $M$?
\end{prob}
\no Recall that the Milnor $\bar{\mu}$--invariants $\{\bar{\mu}_{\mathtt{I};j}(L)\}$ of $L$, are indexed by an ordered tuple $(\mathtt{I};j)=(i_1,i_2,\ldots,i_k;$ $j)$  from the index set $\{1,\ldots,n\}$, there the last index $j$ has a special role (see below). If all the indexes in $(\mathtt{I};j)$ are distinct,  $\{\bar{\mu}_{\mathtt{I};j}\}$ are link homotopy invariants of $L$ and are often referred to simply as {\em Milnor linking numbers} or {\em higher linking numbers}, \cite{Milnor:1954, Milnor:1957}.  The definition $\{\bar{\mu}_{\mathtt{I};j}\}$ begins with coefficients $\mu_{\mathtt{I};j}$ of the Magnus expansion of the $j$th longitude of $L$ in $\pi_1(\R^3-L)$. Then 
%%%%%%%%%%%%%%%%%%%%%%%%%%%%%%%%%%%%%%%%%%%%%%%
\begin{equation}\label{eq:bar-mu-Delta(I)}
\overline{\mu}_{\mathtt{I};j}(L) \equiv \mu_{\mathtt{I};j}(L) \mod\ \Delta_\mu(\mathtt{I};j),\qquad\quad \Delta_\mu(\mathtt{I};j)=\gcd(\Gamma_\mu(\mathtt{I};j)).
\end{equation}
where $\Gamma_\mu(\mathtt{I};j)$ is a certain subset of lower order Milnor invariants, c.f.  \cite{Milnor:1957}. Regarding $\bar{\mu}_{\mathtt{I};j}(L)$ as an element of $\Z_d=\{0,1,\ldots,d-1\}$, $d=\Delta_\mu(\mathtt{I};j)$ (or $\Z$, whenever $d=0$) let us set 
%%%%%%%%%%%%%%%%%%%%%%%%%%%%%%%%%%%%%%%%%%%%%%%
\begin{equation}\label{eq:[mu]}
	\Big[\bar{\mu}_{\mathtt{I};j}(L)\Big]:=\begin{cases}
	\min\Bigl(\bar{\mu}_{\mathtt{I};j},d-\bar{\mu}_{\mathtt{I};j}\Bigr) &\ \text{for $d> 0$},\\
	|\bar{\mu}_{\mathtt{I};j}| & \ \text{for $d= 0$}.
	\end{cases}
\end{equation}
\no Our main result addresses Question \ref{q:F-K} for general $n$--component links (deposing of the Brunnian assumption) as follows.
%%%%%%%%%%%%%%%%%%%%%%%%%%%%%%%%%%%%%%%%%%%%%%%%%%%%%%%%%
\begin{thm}\label{thm:main}
	Let $L$ be an $n$-component link $n\geq 2$ and $\bar{\mu}(L)$ one of its top Milnor linking numbers, then
	\begin{equation}\label{eq:mthm-rop-cr}
	\rop(L)^{\frac 43}\geq \sqrt[3]{n}  \Bigl(\Big[\bar{\mu}(L)\Big]\Bigr)^{\frac{1}{n-1}},\qquad  \Crn(L) \geq \tfrac{1}{3}(n-1) \Bigl(\Big[\bar{\mu}(L)\Big]\Bigr)^{\frac{1}{n-1}}.
	\end{equation}
\end{thm}
\no In the context of Question \ref{q:F-K}, the estimate of Theorem \ref{thm:main} transforms, using \eqref{eq:t(L)-rop(L)}, as follows
\[
 \tau(L)\leq \Bigl(\frac{11}{\sqrt[4]{n}}\Bigr)^{\frac 13} M^{-\frac{1}{4(n-1)}} .
\]
Naturally, Question \ref{q:F-K} can be asked for knots and links and lower bounds in terms of finite type invariants in general. Such questions have been raised for instance in \cite{Cantarella:2009}, where the Bott-Taubes integrals \cite{Bott-Taubes:1994, Volic:2007} have been suggested as a tool for obtaining estimates.
\begin{prob}\label{q:F-K-general}
	Can we find estimates for ropelength of knots/links, in terms of their finite type invariants?
\end{prob}
In the remaining part of this introduction let us sketch the basic idea behind our approach to Question \ref{q:F-K-general}, which relies on the relation between the finite type invariants and the crossing number.

Note that since $\rop(K)$ is scale invariant, it suffices to consider {\em unit thickness knots}, i.e. $K$ together with the unit radius tube neighborhood (i.e. $r(K)=1$). In this setting, $\rop(K)$  just equals the {\em length} $\Ln(K)$ of $K$. From now on we assume unit thickness, unless stated otherwise.  In \cite{Buck-Simon:1999}, Buck and Simon gave the following estimates for  $\Ln(K)$, in terms of the crossing number $\Crn(K)$ of $K$:
%%%%%%%%%%%%%%%%%%%%%%%%%%%%%%%%%%%%%%%%%
\begin{equation}\label{eq:L(K)-BS}
\Ln(K)\geq \Bigr(\frac{4\pi}{11} \Crn(K)\Bigl)^{\frac 34},\qquad \Ln(K)\geq 4\sqrt{\pi \Crn(K)}.
\end{equation}
Clearly, the first estimate is better for knots with large crossing number, while the second one can be sharper for low crossing number knots (which manifests itself for instance in the case of the trefoil). 
Recall that $\Crn(K)$ is a minimal crossing number over all possible knot diagrams of $K$ within the isotopy class of $K$. The estimates in \eqref{eq:L(K)-BS} are a direct consequence of the ropelength bound for the {\em average crossing number}\footnote{i.e. an average of the crossing numbers of diagrams of $K$ over all projections of $K$, see Equation \eqref{eq:aov-acr}.} $\acr(K)$ of $K$ (proven in \cite[Corollary 2.1]{Buck-Simon:1999}) i.e. 
%%%%%%%%%%%%%%%%%%%%%%%%%%%%%%%%%%%%%%%%%
\begin{equation}\label{eq:rop-BS-acr}
\Ln(K)^{\frac 43}\geq \frac{4\pi}{11} \acr(K),\qquad \Ln(K)^{2}\geq 16\pi \acr(K).
\end{equation}

In Section \ref{S:links_ropelength}, we obtain 
an analog of \eqref{eq:L(K)-BS} for $n$--component links ($n\geq 2$) in terms of the {\em pairwise crossing number}\footnote{see \eqref{eq:PCr(L)} and Corollary \ref{cor:rop-PCr}, generally $\PCrn(L)\leq \Crn(L)$, as the individual components can be knotted.}
$\PCrn(L)$, as follows
\begin{equation}\label{eq:PCr-rop-3/4-2}
\Ln(L)\geq  \frac{1}{\sqrt{n-1}}\Bigl(\tfrac 32 \PCrn(L)\Bigr)^{\frac 34},\qquad \Ln(L)\geq \frac{n\sqrt{16\pi}}{\sqrt{n^2-1}}\Bigl(\PCrn(L)\Bigr)^{\frac 12}.
\end{equation}
 For low crossing number knots, the Buck and Simon bound \eqref{eq:L(K)-BS} was further improved by Diao\footnote{More precisely: $16\pi \Crn(K) \leq \Ln(K)(\Ln(K)-17.334)$ \cite{Diao:2003}.} \cite{Diao:2003},
\begin{equation}\label{eq:L(K)-D}
\Ln(K)\geq \tfrac{1}{2} \bigl(d_0 +\sqrt{d^2_0+64\pi \Crn(K)}\bigr),\qquad d_0=10-6(\pi+\sqrt{2})\approx 17.334.
\end{equation}

On the other hand, there are well known estimates for $\Crn(K)$ in terms of finite type invariants of knots. For instance,
\begin{equation}\label{eq:c_2-cr-LW-PV}
\tfrac{1}{4}\Crn(K)(\Crn(K)-1)+\tfrac{1}{24}\geq |c_2(K)|, \qquad  \tfrac{1}{8}(\Crn(K))^2\geq |c_2(K)|.
\end{equation}
Lin and Wang \cite{Lin-Wang:1996} considered the second coefficient of the Conway polynomial $c_2(K)$ (i.e. the first nontrivial type $2$ invariant of knots) and proved the first bound in \eqref{eq:c_2-cr-LW-PV}. The second estimate of \eqref{eq:c_2-cr-LW-PV} can be found in Polyak--Viro's work \cite{Polyak-Viro:2001}.
Further, Willerton, in his thesis \cite{Willerton:2002} obtained estimates for the ``second'', after $c_2(K)$, finite type invariant $V_3(K)$ of type $3$, as 
\begin{equation}\label{eq:V_3-cr}
\tfrac{1}{4}\Crn(K)(\Crn(K)-1)(\Crn(K)-2)\geq |V_3(K)|.
\end{equation}
In the general setting, Bar-Natan \cite{Bar-Natan:1995} shows that if $V(K)$ is a type $n$ invariant then $|V(K)| = O(\Crn(K)^n)$. All these results rely on the arrow diagrammatic formulas for Vassiliev invariants developed in the work of Goussarov, Polyak and Viro \cite{Goussarov-Polyak-Viro:2000}.

Clearly, combining \eqref{eq:c_2-cr-LW-PV} and \eqref{eq:V_3-cr} with \eqref{eq:L(K)-BS} or \eqref{eq:L(K)-D}, immediately yields lower bounds for  ropelength in terms of a given Vassiliev invariant. One may take these considerations one step further and extend the above estimates to  the case of the $2n$\textsuperscript{th} coefficient of the Conway polynomial $c_{2n}(K)$, with the help of arrow diagram formulas for $c_{2n}(K)$, obtained recently in \cite{Chmutov-Duzhin-Mostovoy:2012, Chmutov-Khoury-Rossi:2009}. In Section \ref{S:knots}, we follow the Polyak--Viro's argument of \cite{Polyak-Viro:2001} to obtain
 %%%%%%%%%%%%%%%%%%%%%%%%%%%%%%%%%%%%%%%%%%%%%%%%%%%%%
 \begin{thm}\label{thm:conway}
 	Given a knot $K$, we have the following crossing number estimate
 	\begin{equation}\label{eq:crn_conway}
 	\Crn(K) \geq \bigl((2n)! |c_{2n}(K)|\bigl)^{\frac{1}{2n}}\geq\frac{2n}{3} |c_{2n}(K)|^{\frac{1}{2n}}.
 	\end{equation}
 \end{thm}
 \no Combining \eqref{eq:crn_conway} with Diao's lower bound \eqref{eq:L(K)-D} one obtains
 %%%%%%%%%%%%%%%%%
 \begin{cor}
 	For a unit thickness knot $K$,
 \begin{equation}\label{eq:rop_conway}
 \Ln(K) \geq \tfrac{1}{2} \Bigl(d_0 +\bigl(d^2_0+ \tfrac{128 n \pi}{3} |c_{2n}(K)|^{\frac{1}{2n}}\bigr)^{\frac 12}\Bigr).
 \end{equation}
 \end{cor}
 %%%%%%%%%%%%%%
 Recall that a somewhat different approach to ropelength estimates is presented in  \cite{Cantarella-Kusner-Sullivan:2002}, where the authors introduce a cone surface technique, which combined with the asymptotic crossing number, $\Ac(K)$, bound of Freedman and He, \cite{Freedman-He:1991} gives
 \begin{equation}\label{eq:L(K)-AC-g(K)}
 \Ln(K)\geq 2\pi+2\pi\sqrt{\Ac(K)},\qquad \Ln(K)\geq 2\pi+2\pi\sqrt{2 g(K)-1},
 \end{equation}
 where the second bound follows from the knot genus estimate of \cite{Freedman-He:1991}: $\Ac(K)\geq 2 g(K)-1$.
 
  When comparing Estimate \eqref{eq:L(K)-AC-g(K)} and \eqref{eq:rop_conway}, in favor of Estimate \eqref{eq:rop_conway}, we may consider a family of {\em pretzel knots}: $P(a_1,\ldots,a_n)$ where $a_i$ is the number of signed crossings in the $i$th tangle of the diagram, see Figure \ref{fig:pretzel_knots}. Additionally, for a diagram $P(a_1,\ldots,a_n)$ to represent a knot one needs to assume either both $n$ and all $a_i$ are odd or one of the $a_i$ is even, \cite{Kawauchi:1996}. 
  
 %%%%%%%%%%%%%%%%%%%%%%%%%%%%%%%%%%%%%%%%%
 \begin{wrapfigure}{l}{7cm}
 	\centering
 	\includegraphics[width=.35\textwidth]{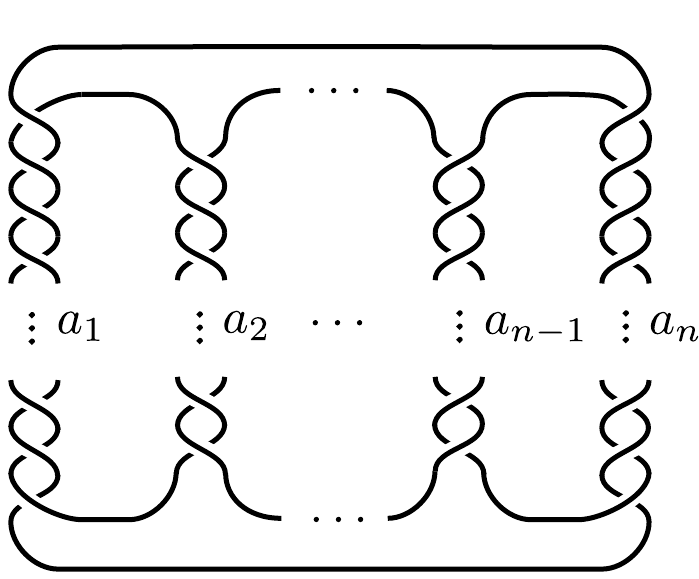}
 	\caption{$P(a_1,\ldots,a_n)$ pretzel knots.}\label{fig:pretzel_knots}
 \end{wrapfigure}
 Genera of selected subfamilies of pretzel knots are known, for instance \cite[Theorem 13]{Manchon:2012} implies  
 \begin{align*}
 \qquad & \qquad & \qquad & \qquad & g(P(a,b,c)) =1,\\
 \qquad & \qquad & \qquad & \qquad & c_2(P(a,b,c))=\frac 14 (ab+ac+bc+1),
 \end{align*}
 where $a$, $b$, $c$ are odd integers with the same sign
 (for the value of $c_2(P(a,b,c))$ see a table in \cite[p. 390]{Manchon:2012}). Concluding, the lower bound in  
 \eqref{eq:rop_conway} can be made arbitrary large by letting $a,b,c\to+\infty$, while the lower bound 
 in \eqref{eq:L(K)-AC-g(K)} stays constant for any values of $a$, $b$, $c$, under consideration. Yet another\footnote{out of a few such examples given in \cite{Manchon:2012}.} example of a family of pretzel knots with constant genus one and arbitrarily large $c_2$--coefficient is   $D(m,k)=P(m,\varepsilon,\stackrel{|k|-\text{times}}{\ldots},\varepsilon)$, $m>0$, $k$, where $\varepsilon=\frac{k}{|k|}$ is the sign of $k$ (e.g. $D(3,-2)=P(3,-1,-1)$). For any such $D(m,k)$, we have $c_2(D(m,k))=\frac{mk}{4}$.
\begin{rem}
{\rm 	A natural question can be raised about the the reverse situation: Can we find  a family of knots with constant $c_{2n}$-coefficient (or any finite type invariant, see Remark \ref{rem:method-extended}), but arbitrarily large genus? For instance, there exist knots with $c_2=0$ and nonzero genus (such as $8_2$), in these cases \eqref{eq:L(K)-AC-g(K)} still provides a nontrivial lower bound.} 
\end{rem}

The paper is structured as follows: Section \ref{S:knots} is devoted to a review of 
arrow polynomials for finite type invariants and Kravchenko-Polyak tree invariants in particular, it also contains the proof Theorem \ref{thm:conway}.  Section \ref{S:links_ropelength} contains information on the average overcrossing number for links and link ropelength estimates analogous to the ones obtained by Buck and Simon \cite{Buck-Simon:1999} (see Equation \eqref{eq:rop-BS-acr}). The proof of Theorem \ref{thm:main} is presented in Section \ref{S:proof-main}, together with final comments and remarks.

%%%%%%%%%%%%%%%%%%%%%%%%%%%%%%%%%%%%%%%%%%%%%%%%%%%%%%%%%%%
\section{Acknowledgements}
Both authors acknowledge the generous support of NSF DMS 1043009 and DARPA YFA N66001-11-1-4132 during the years 2011-2015. 
The first author extends his thanks to Jason Cantarella (who pointed out to his work in \cite{Cantarella:2009} used in Lemma \ref{lem:aov-bound}), Jason Parsley and Clayton Shonkwiler for stimulating  conversations about ropelength during the {\em Entanglement and linking} workshop in Pisa, 2012. %Jason Cantarella pointed the first author to the rearrangement trick of \cite{Cantarella:2009} used in Lemma \ref{lem:aov-bound} (during the {\em Entanglement and linking} workshop in Pisa 2012). 
Both authors wish to thank the anonymous referee, for a very detailed report, which improved the quality of the paper.

The results of this paper grew out of considerations in the second author's doctoral thesis \cite{Michaelides:2015}, where a weaker versions of the estimates (in the Borromean case) were obtained.

%%%%%%%%%%%%%%%%%%%%%%%%%%%%%%%%%%%%%%%%%
\begin{figure}[ht]
	\includegraphics[width=.6\textwidth]{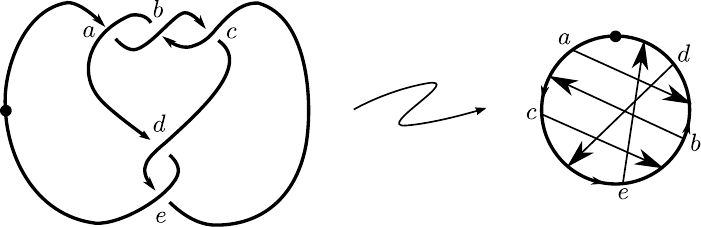}
	\caption{$5_2$ knot and its Gauss diagram (all crossings are positive).}\label{fig:5_2-gauss}
\end{figure}

%%%%%%%%%%%%%%%%%%%%%%%%%%%%%%%%%%%%%%%%%%%%%%%%%%%%%%%%%%%
\section{Arrow polynomials and finite type invariants}\label{S:knots}

Recall from \cite{Chmutov-Duzhin-Mostovoy:2012}, the {\em Gauss diagram} of a knot $K$ is a way of representing signed overcrossings in a knot diagram, by arrows based on a circle ({\em Wilson loops}, \cite{Bar-Natan:1991}) with signs encoding the sign of the crossings (see Figure \ref{fig:5_2-gauss} showing the $5_2$ knot and its Gauss diagram). More precisely, the Gauss diagram $G_K$ of a knot $K:S^1\longrightarrow \R^3$ is constructed by marking pairs of points in the domain $S^1$, endpoints of a correspoinding arrow in $G_K$,  which are mapped to crossings in a generic planar projection of $K$. The arrow always points from the under-- to the over--crossing and the orientation of the circle $S^1$ in $G_K$ agrees with the orientation of the knot.

 Given a Gauss diagram $G$, the {\em arrow polynomials} of \cite{Goussarov-Polyak-Viro:2000, Polyak-Viro:1994} are defined simply as a signed count of selected subdiagrams in $G$. For instance the second coefficient of the Conway polynomial $c_2(K)$ is given 
by the signed count of $\vvcenteredinclude{.2}{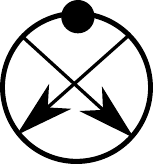}$ in $G$, denoted as 
\begin{equation}\label{eq:c_2-arrow}
 c_2(K)=\langle\vvcenteredinclude{.2}{c2-arrow.pdf}, G\rangle=\sum_{\phi:\vvcenteredinclude{.1}{c2-arrow.pdf}\longrightarrow G} \sign(\phi),\qquad \sign(\phi)=\prod_{\alpha\in \vvcenteredinclude{.1}{c2-arrow.pdf}} \sign(\phi(\alpha)),
\end{equation}
where the sum is over all basepoint preserving graph embeddings $\{\phi\}$ of $\vvcenteredinclude{.2}{c2-arrow.pdf}$ into $G$, and the sign is a product of signs of corresponding arrows in $\phi(\vvcenteredinclude{.2}{c2-arrow.pdf})\subset G$. For example in the Gauss diagram of $5_2$ knot in Figure \ref{fig:5_2-gauss}, there are two possible embeddings of $\vvcenteredinclude{.2}{c2-arrow.pdf}$ into the diagram. One matches the pair of arrows $\{a,d\}$ and another pair $\{c,d\}$, since all crossings are positive we obtain $c_2(5_2)=2$.
\begin{figure}[!ht] 
	\centering
	\includegraphics[width=0.4\textwidth]{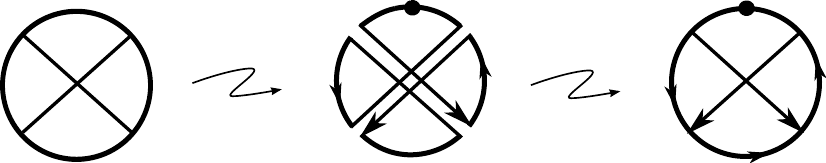}
	\caption{Turning a one-component chord diagram with a base point into an arrow diagram}\label{fig:chord_to_arrow}
\end{figure} 

For other even coefficients of the Conway polynomial, the work in \cite{Chmutov-Khoury-Rossi:2009} provides  the following recipe for their arrow polynomials. Given $n>0$, consider any chord diagram $D$, on a single circle component with $2n$ chords, such as $\vvcenteredinclude{.25}{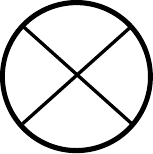}$, $\vvcenteredinclude{.25}{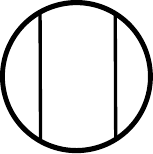}$, $\vvcenteredinclude{.25}{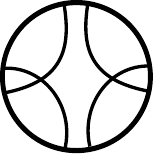}$. A chord diagram $D$ is said to be a {\em $k$-component diagram}, if after parallel doubling of each chord according to  \vcenteredinclude{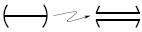}, the resulting curve will have $k$--components. For instance $\vvcenteredinclude{.25}{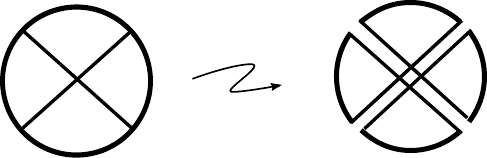}$ is a $1$--component diagram and $\vvcenteredinclude{.25}{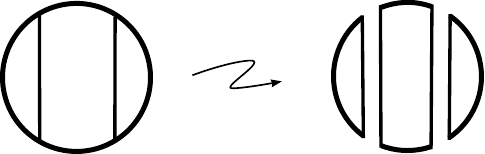}$ is a $3$--component diagram.
For the coefficients $c_{2n}$, only one component diagrams will be of interest and we turn a one-component chord diagram with a base point into an arrow diagram according to the following rule \cite{Chmutov-Khoury-Rossi:2009}:
\begin{quote}
	{\em Starting from the base point we move along the diagram with doubled chords. During this journey we pass both copies of each chord in opposite directions. Choose an arrow on each chord which corresponds to the direction of the first passage of the copies of the chord (see Figure \ref{fig:chord_to_arrow} for the illustration).}
\end{quote}
 We call, the arrow diagram obtained according to this method, the {\em  ascending arrow diagram} and denote by $C_{2n}$ the sum of all based one-component ascending arrow diagrams with $2n$ arrows. For example $\vvcenteredinclude{.3}{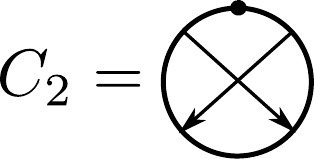}$ and $C_{4}$ is shown below (c.f. \cite[p. 777]{Chmutov-Khoury-Rossi:2009}).
\[
\vvcenteredinclude{.31}{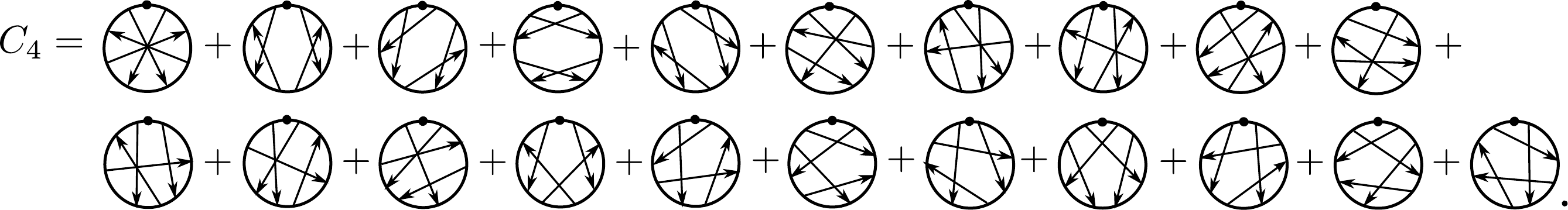}
\]
In \cite{Chmutov-Khoury-Rossi:2009}, the authors show for $n \geq 1$, that the $c_{2n}(K)$ coefficient  of the Conway polynomial of $K$ equals 
	\begin{equation}\label{eq:conway_coefficients}
	c_{2n}(K) = \langle C_{2n}, G_K \rangle.
	\end{equation}
%%%%%%%%%%%%%%%%%%%%%%%%%%%%%%%%%%%%%%%%%%%%%%%%%%%%%
  \begin{repthm}{thm:conway}
  	Given a knot $K$, we have the following crossing number estimate
  	\begin{equation}\label{eq2:crn_conway}
  	\Crn(K) \geq \bigl((2n)! |c_{2n}(K)|\bigl)^{\frac{1}{2n}}\geq\tfrac{2n}{3} |c_{2n}(K)|^{\frac{1}{2n}}.
  	\end{equation}
  \end{repthm}
 %%%
\begin{proof}
	\no Given $K$ and its Gauss diagram $G_K$, let $X=\{1,2,\ldots, \crn(K)\}$ index arrows of $G_K$ (i.e. crossings of a diagram of $K$ used to obtain $G_K$). For diagram term $A_i$ in the sum $C_{2n}=\sum_i A_i$, an embedding $\phi:A_i\longmapsto G_K$ covers a certain $2n$ element subset of crossings in $X$ that we denote by $X_{\phi}(i)$. Let $\mathcal{E}(i;G_K)$ be the set of all possible embeddings  $\phi:A_i\longmapsto G_K$, and 
	\[
	\mathcal{E}(G_K)=\bigsqcup_i \mathcal{E}(i;G_K).
	\]
	Note that $X_\phi(i)\neq X_\xi(j)$ for $i\neq j$ and $X_\phi(i)\neq X_\xi(i)$ for $\phi\neq \xi$, thus for each $i$ we have an injective map
	\[
	F_i:\mathcal{E}(i;G_K)\longmapsto \mathcal{P}_{2n}(X),\qquad F_i(\phi)=X_\phi(i),
	\]
	where $\mathcal{P}_{2n}(X)=\{\text{$2n$--element subsets of $X$}\}$. $F_i$ extends in an obvious way to the whole disjoint union $\mathcal{E}(G_K)$, as $F:\mathcal{E}(G_K)\longrightarrow \mathcal{P}_{2n}(X)$, $F=\sqcup_i F_i$ and remains injective. In turn, for every $i$ we have 
	\[
	|\langle A_i,G_K\rangle|\leq \# \mathcal{E}(i;G_K)
	\]
	and therefore 
	\[
	  |\langle C_{2n},G_K\rangle|\leq \# \mathcal{E}(G_K)< \#\mathcal{P}_{2n}(X)={\crn(K) \choose 2n}.
	\]
	If $\crn(K)<2n$ then the left hand side vanishes. 
%Further, each arrow in $G_K$ indexed by $X$, either agrees with the orientation of the Wilson loop (then we say it is a {\em right} arrow) or not (then it is a {\em left} arrow), thus $X=L\cup R$, $L\cap R=\varnothing$ where $L$ is the subset of left arrows and $R$ the subset or right arrows with cardinalities $r=\# R$ and $l=\# L$, $l+r=\crn(K)$. Since each arrow diagram $A_i$ has exactly $n$ right arrows and $n$ left arrows we must have 
%	\[
%	\sum_i\# \mathcal{E}(i;G_K)\leq {\crn(K)-l \choose n} {l \choose n},\qquad n\leq l\leq  \crn(K)-n.
%	\] 
%	The left hand side is maximized for $l=\bigl[\frac{\crn(K)}{2}\bigr]$ and thus we obtain
%	\[
%	2\,|c_{2n}(K)|\leq {l \choose n}^2\leq \Bigl(\frac{l^n}{n!}\Bigr)^2,
%	\]
%which proves the first inequality in \eqref{eq2:crn_conway}. 
Since ${\crn(K) \choose 2n}\leq \frac{\crn(K)^{2n}}{(2n)!}$, we obtain
	\[
		|c_{2n}(K)|\leq  \frac{\crn(K)^{2n}}{(2n)!}\quad \Rightarrow\quad 	\bigl((2n)! |c_{2n}(K)|\bigl)^{\frac{1}{2n}}\leq \crn(K),
	\]
which gives the first part of \eqref{eq2:crn_conway}. Using the following upper lower bound for $m!$ (Stirling's approximation, \cite{Abramowitz:1964aa})
\[
m!\geq \sqrt{2\pi} m^{m+\frac 12} e^{-m},
\]
applying , $e^{-1}\geq \frac{1}{3}$, $\bigl(\sqrt{2\pi}\bigr)^{\frac{1}{m}}\geq 1$, $\bigl(m^{m+\frac 12}\bigr)^{\frac{1}{m}}\geq m  \bigl(\sqrt{m}\bigr)^{\frac{1}{m}}\geq m$, yields
\begin{equation}\label{eq:stirling2} 
\bigl(m!\bigr)^{\frac{1}{m}}\geq \bigl(\sqrt{2\pi} (m)^{m+\frac 12} e^{-m}\bigr)^{\frac{1}{m}} \geq \frac{m}{3},
\end{equation}
for $m=2n$ one obtains the second part of \eqref{eq2:crn_conway}.
\end{proof}
%%%%%%%%%%%%%%%%%%%%%%%%%%%%%%%%%%%
\begin{figure}[ht]
	\includegraphics[width=.08\textwidth]{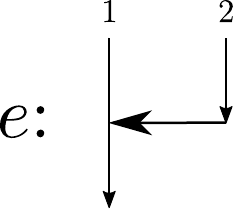}\qquad\qquad \includegraphics[width=.08\textwidth]{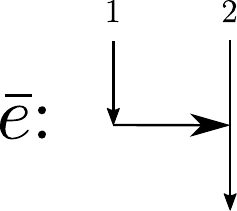}\qquad\qquad\qquad \includegraphics[width=.15\textwidth]{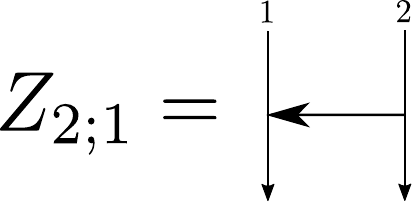}
	\caption{Elementary trees $e$ and $\bar{e}$ and the $Z_{2;1}$ arrow polynomial. }\label{fig:elem-tree} 
\end{figure}	

 Next, we turn to arrow polynomials for Milnor linking numbers. In \cite{Kravchenko-Polyak:2011}, Kravchenko and Polyak introduced tree invariants of string links and established their relation to Milnor linking numbers via the skein relation of \cite{Polyak:2000}. In the recent paper, the authors\footnote{consult \cite{Kotorii:2013} for a related result.} \cite{Komendarczyk-Michaelides:2016} showed that the arrow polynomials of Kravchenko and Polyak, applied to Gauss diagrams of closed based links, yield certain $\bar{\mu}$--invariants (as defined in \eqref{eq:bar-mu-Delta(I)}). 
For a purpose of the proof of Theorem \ref{thm:main}, it suffices to give a recursive definition, provided  in \cite{Komendarczyk-Michaelides:2016}, for the arrow polynomial of $\bar{\mu}_{23\ldots n;1}(L)$ denoted by $Z_{n;1}$. Changing the convention, adopted for knots, we follow \cite{Kravchenko-Polyak:2011}, \cite{Komendarczyk-Michaelides:2016} and use vertical segments (strings) oriented downwards in place of circles (Wilson loops) as components. 
\begin{figure}[!ht] 
	\begin{overpic}[width=.5\textwidth]{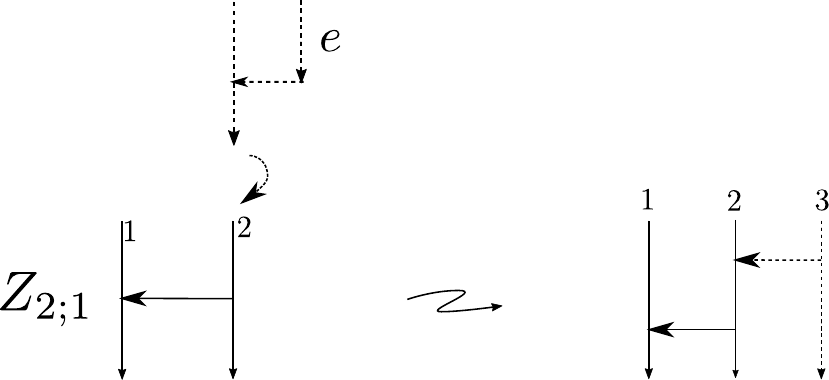}
	\end{overpic}
	\caption{Obtaining a term in $Z_{3;1}$ via stacking $e$ on the second component of $Z_{2;1}$, i.e. $Z_{2;1}\prec_2 e$. }\label{fig:Z_2;1-to-term-Z_3;1}
\end{figure}
The polynomial $Z_{n;1}$ is obtained inductively from~$Z_{n-1;1}=\sum_k \pm A_k$ by expanding each term $A_k$ of $Z_{n;1}$ through stacking elementary tree diagrams $e$ and $\bar{e}$, shown in Figure \ref{fig:elem-tree}, the sign of a resulting term is determined accordingly. The stacking operation is denoted by $\prec_i$, where $i=1,\ldots, n$ tells which component is used for stacking. Figure \ref{fig:Z_2;1-to-term-Z_3;1} shows  $Z_{2;1}\prec_2 e$. The inductive procedure is defined as follows:
\begin{itemize}
	\item[$(i)$]  $Z_{2;1}$ is shown in Figure \ref{fig:elem-tree}(right).
	\item[$(ii)$] For each term $A_k$ in $Z_{n-1;1}$ produce terms in $Z_{n;1}$ by stacking\footnote{Note that $\bar{e}$ is not allowed to be stacked on the first component.} $e$ and $\bar{e}$ on each component i.e. $A_k\prec_i e$ for $i=1,\ldots, n$ and $A_k\prec_i \bar{e}$ for $i=2,\ldots, n$, see Figure \ref{fig:Z_2;1-to-term-Z_3;1}. Eliminate isomorphic (duplicate) diagrams.
	\item[$(iii)$] The sign of each term in $Z_{n;1}$ equals to $(-1)^q$, $q=$number of arrows pointing to the right.
\end{itemize} 

As an example consider $Z_{3;1}$; we begin with the initial tree $Z_{2;1}$, and expand by stacking $e$ and $\bar{e}$ on the strings of $Z_{2;1}$, this is shown in Figure \ref{fig:Z_2;1-to-Z_3;1}, we avoid stacking $\bar{e}$ on the first component (called the {\em trunk}, \cite{Komendarczyk-Michaelides:2016}). Thus $Z_{3;1}$ is obtained as $A+B-C$, where $A=Z_{2;1}\prec_2 e$, $B=Z_{2;1}\prec_1 e$, and $C=Z_{2;1}\prec_2 \bar{e}$. 
\begin{figure}[!ht] 
	\begin{overpic}[width=.5\textwidth]{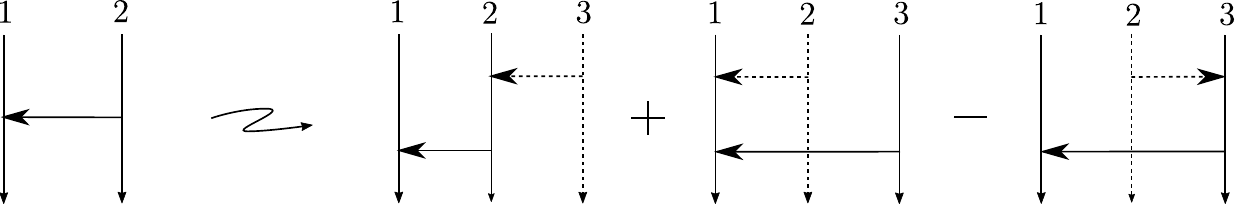}
	\end{overpic}
	\caption{$Z_{3;1}=A+B-C$ obtained from $Z_{2;1}$ via $(i)$--$(iii)$. }\label{fig:Z_2;1-to-Z_3;1}
\end{figure}	
\no Given $Z_{n;1}$, the main result of \cite{Komendarczyk-Michaelides:2016} (see also \cite{Kotorii:2013} for a related result) yields the following formula
\begin{equation}\label{eq:bar-mu-Z_n;1}
\overline{\mu}_{n;1}(L) \equiv \langle Z_{n;1}, G_L\rangle \mod\ \Delta_\mu(n;1).
\end{equation} 
where $\overline{\mu}_{n;1}(L):=\overline{\mu}_{2\ldots n;1}(L)$, $G_L$ a Gauss diagram of an $n$--component link $L$, and the indeteminacy $\Delta_\mu(n;1)$ is defined in \eqref{eq:bar-mu-Delta(I)}. 
Recall that $\langle Z_{n;1}, G_L\rangle=\sum_k \pm \langle A_k, G_L\rangle$ ($Z_{n;1}=\sum_k \pm A_k$) where  $\langle A_k, G_L\rangle=\sum_{\phi:A_k\longrightarrow G_L} \sign(\phi)$ is a signed count of subdiagrams isomorphic to $A_k$ in $G_L$.

 For $n = 2$, we obtain the usual linking number
\begin{equation}\label{eq:lk-gauss}
\bar{\mu}_{2;1}(L)=\langle Z_{2;1}, G_L\rangle = \Bigl\langle\vvcenteredinclude{0.23}{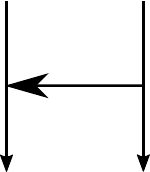}, G_L\Bigr\rangle.
\end{equation}
For $n=3$ and $n=4$ the arrow polynomials can be obtained following the stacking procedure 
\[
\begin{split}
 \bar{\mu}_{3;1}(L) & =\langle Z_{3;1}, G_L\rangle \mod \gcd\{\bar{\mu}_{2;1}(L), \bar{\mu}_{3;1}(L), \bar{\mu}_{3;2}(L)\},\\
 &\quad  Z_{3;1} = \vvcenteredinclude{0.3}{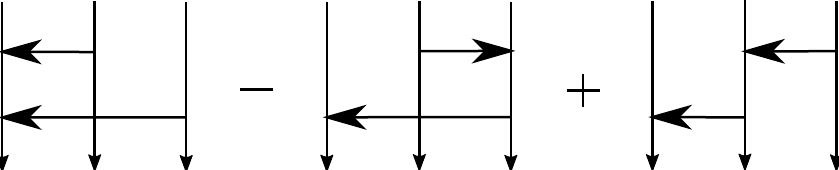},
\end{split}
\]
and
\[
\begin{split}
\bar{\mu}_{4;1}(L) & =\langle Z_{4;1}, G_L\rangle \mod \Delta_\mu(4;1),\\
 &\quad  Z_{4;1} = \vvcenteredinclude{0.3}{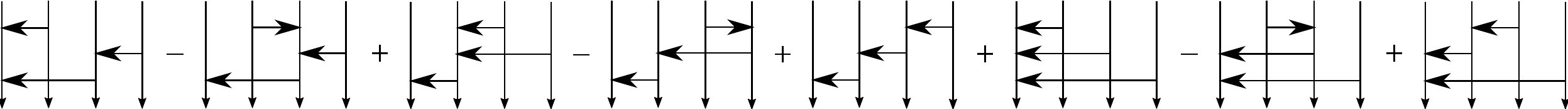} \\
&\qquad\qquad\vvcenteredinclude{0.3}{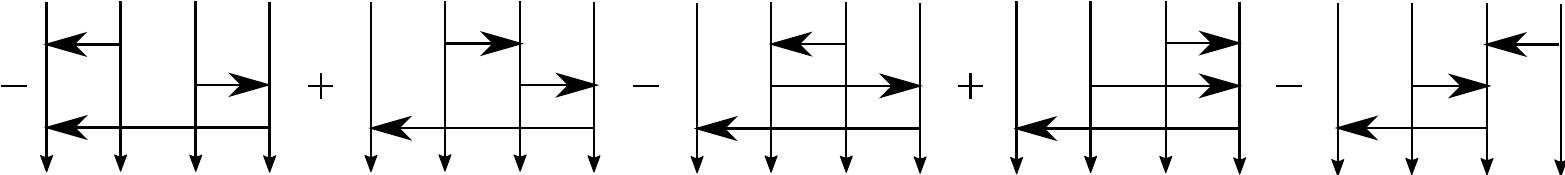}\ .
\end{split}
\]
\smallskip

Given a formula for $\bar{\mu}_{n;1}(L)=\bar{\mu}_{23\ldots n;1}(L)$ all remaining $\bar{\mu}$--invariants with distinct indices can be obtained from the following permutation identity (for $\sigma\in \Sigma(1,\ldots,n)$)
\begin{equation}\label{eq:mu-symmetry}
\bar{\mu}_{\sigma(2)\sigma(3)\ldots\sigma(n);\sigma(1)}(L)=\bar{\mu}_{23\ldots n;1}(\sigma(L)),\qquad \sigma(L)=(L_{\sigma(1)},L_{\sigma(2)},\ldots, L_{\sigma(n)}).
\end{equation}
By \eqref{eq:bar-mu-Z_n;1}, \eqref{eq:mu-symmetry} and \eqref{eq:bar-mu-Delta(I)} we have
\begin{equation}\label{eq:bar-mu-sigma-Z}
 \bar{\mu}_{\sigma(2)\sigma(3)\ldots\sigma(n);\sigma(1)}(L)= \langle \sigma(Z_{n;1}), G_L\rangle \mod\ \Delta_\mu(\sigma(2)\sigma(3)\ldots\sigma(n);\sigma(1)),
\end{equation}
where $\sigma(Z_{n;1})$ is the arrow polynomial obtained from $Z_{n;1}$ by permuting the strings according to $\sigma$.
\begin{rem}\label{rem:mu-cyclic}
	{\em
	 One of the properties of $\bar{\mu}$--invariants is their cyclic symmetry,  \cite[Equation (21)]{Milnor:1957}, i.e. given a cyclic permutation $\rho$, we have
	 \[
	  \bar{\mu}_{\rho(2)\rho(3)\ldots\rho(n);\rho(1)}(L)=\bar{\mu}_{23\ldots n;1}(L).
	 \]
	}
\end{rem}
%%%%%%%%%%%%%%%%%%%%%%%%%%%%%%%%%%%%%%%%%%%%%%%%%%%%%%%%%%%%%%%
\section{Overcrossing number of links}\label{S:links_ropelength}

We will denote by $D_L$ a regular diagram of a link $L$, and by $D_L(v)$, the diagram obtained by the projection of $L$ onto the plane normal to a vector\footnote{unless otherwise stated we assume that $v$ is generic and thus $D_L(v)$ to be a regular diagram.} $v\in S^2$. For a pair of components $L_i$ and $L_j$ in $L$, define the {\em overcrossing number}  in the diagram and the {\em pairwise crossing number} of components $L_i$ and $L_j$ in $D_L$ i.e.
%%%%%%%%%%%%%%%%%%%%%%%%%%%%%%%%%%%%%%%%%%%%%%
\begin{equation}\label{eq:ov-cr}
\begin{split}
\mathsf{ov}_{i,j}(D_L) & =\{\text{number of times $L_i$ overpasses $L_j$ in $D_L$}\}.\\
\mathsf{cr}_{i,j}(D_L) & =\{\text{number of times $L_i$ overpasses and underpasses $L_j$ in $D_L$}\}\\
& =\mathsf{ov}_{i,j}(D_L)+\mathsf{ov}_{j,i}(D_L)=\mathsf{cr}_{j,i}(D_L).
\end{split}
\end{equation}
 In the following, we also use the {\em average overcrossing number} and {\em  average pairwise crossing number} of components $L_i$ and $L_j$ in $L$, defined as an average over all $D_L(v)$, $v\in S^2$, i.e.
%%%%%%%%%%%%%%%%%%%%%%%%%%%%%%%%%%%%%%%%%%%%%%
\begin{equation}\label{eq:aov-acr}
\mathsf{aov}_{i,j}(L)=\frac{1}{4\pi}\int_{S^2}  \mathsf{ov}_{i,j}(v)\, dv,\qquad \mathsf{acr}_{i,j}(L)=\frac{1}{4\pi}\int_{S^2}  \mathsf{cr}_{i,j}(v)\, dv=2\, \mathsf{aov}_{i,j}(L)
\end{equation}
In the following result is based on the work in \cite{Cantarella:2009} and \cite{Buck-Simon:1999}, the idea of using the rearrangement inequality comes from \cite{Cantarella:2009}. 
\begin{lem}\label{lem:aov-bound}
	Given a unit thickness link $L$, and any $2$--component sublink $(L_i,L_j)$:
	\begin{equation}\label{eq:aov-bound}
	\min\bigl(\ell_i\ell^{\frac 13}_j,\ell_j\ell^{\frac 13}_i\bigr)\geq 3\,\mathsf{aov}_{i,j}(L),\qquad \ell_i\ell_j\geq 16\pi\,\mathsf{aov}_{i,j}(L),
	\end{equation}
	for $\ell_i=\Ln(L_i)$, $\ell_j=\Ln(L_j)$, the
	length of $L_i$ and $L_j$ respectively.
\end{lem}

\begin{proof}
	Consider the Gauss map of $L_i=L_i(s)$ and $L_j=L_j(t)$:
	\begin{equation*}
	F_{i,j}:S^1\times S^1\longmapsto \Cnf_2(\R^3)\longmapsto S^2,\quad F_{i,j}(s,t)=\frac{L_i(s)-L_j(t)}{||L_i(s)-L_j(t)||}.
	\end{equation*}
	
	\no If $v\in S^2$ is a regular value of $F_{i,j}$ (which happens for the set of full measure on $S^2$) then
	\begin{equation*}
	\mathsf{ov}_{i,j}(v)=\#\{\text{points in $F^{-1}_{i,j}(v)$} \}.
	\end{equation*}
	i.e. $\mathsf{ov}_{i,j}(v)$ stands for number of times the $i$--component of $L$ passes over the $j$--component, in the projection of $L$ onto the plane in $\R^3$ normal to $v$. As a direct consequence of Federer's coarea formula \cite{Federer:1969} (see e.g. \cite{Michaelides:2015} for a proof)
	\begin{align} \label{eq:gauss-linking-int}
	\int_{L_i\times L_j} |F^\ast_{i,j} \omega| & =\frac{1}{4\pi}\int_{S^1\times S^1} \frac{|\langle  L_i(s)-L_j(t),L'_i(s), L'_j(t) \rangle|}{||L_i(s)-L_j(t)||^3}  ds \, dt\\
	& =\frac{1}{4\pi}\int_{S^2}  \mathsf{ov}_{i,j}(v)\, dv, \label{eq:avg-overcrossing}
	\end{align}
	where $\omega = \frac{1}{4\pi}\bigl(x \,dy \wedge dz - y \, dx \wedge dz + z\,  dx \wedge dy\bigr)$ is the normalized area form on the unit sphere in $\R^3$ and 
	\[
	 \langle  v, w, z \rangle:=\det\bigl(v,w,z\bigr),\qquad\text{for}\quad v,w,z\in \R^3.
	\] 
	Assuming the arc--length parametrization by $s \in [0, \ell_i]$ and $t \in [0, \ell_j]$ of the components we have $\|L'_i (s)\|=\|L'_j(t)
	\|=1$ and therefore: 
	\begin{equation}\label{eq:quarter}
	\Bigl|\frac{\langle L_i(s)-L_j(t),  L'_i(s),L'_j(t) \rangle}{||L_i(s)-L_j(t)||^3}\Bigl|\leq \frac{1}{||L_i(s)-L_j(t)||^2}	
	\end{equation}
	\no Combining Equations \eqref{eq:avg-overcrossing} and \eqref{eq:quarter} yields
	\begin{equation}\label{eq:overcrossing}
	\int^{\ell_j}_0\int^{\ell_i}_0 \frac{1}{||L_i(s) - L_j(t)||^2}ds\,dt=\int^{\ell_j}_0 I_i(L_j(t))dt,
	\end{equation}
	where 
	\[
	 I_i(p)=\int^{\ell_i}_0 \frac{1}{||L_i(s) - p||^2}ds=\int^{\ell_i}_0 \frac{1}{r(s)^2}ds,\quad r(s)=\|L_i(s) - p\|,
	\]
	is often called {\em illumination} of $L_i$ from the point $p\in \R^3$, c.f. \cite{Buck-Simon:1999}. Following the approach of \cite{Buck-Simon:1999}, and \cite{Cantarella:2009} we estimate $I_i(t)=I_i(p)$ for $p=L_j(t)$. Denote  by $B_a(p)$ the ball at $p=L_j(t)$ of radius $a$, and $s(z)$ the length of a portion of $L_i$ within the spherical shell: $Sh(z)=B_z(p)\setminus B_1(p)$, $z\geq 1$. Note that, because the distance between $L_i$ and $L_j$ is at least $2$, the unit thickness tube about $L_i$ is contained entirely in $Sh(z)$ for big enough $z$. Clearly, $s(z)$ is nondecreasing. Since the volume of a unit thickness tube of length $a$ is $\pi a$, comparing the volumes we obtain
	\begin{equation}\label{eq:vol-tube-shell}
	\begin{split}
	\pi s(z) & \leq \operatorname{Vol}(Sh(z))=\tfrac{4}{3}\pi\bigl(z^3-1^3\bigr), \ \text{and}\\
	s(z) & \leq \tfrac{4}{3}\, z^3, \qquad \text{for}\quad z\geq 1.
	\end{split}
	\end{equation}
	Next, using the monotone rearrangement $\bigl(\frac{1}{r^2}\bigr)^\ast$ of $\frac{1}{r^2}$,  (Remark \ref{rem:rearrangement}):
	 \begin{equation}\label{eq:rng-ineq}
	 \Bigl(\frac{1}{r^2}\Bigr)^\ast(s)\leq (\tfrac{4}{3})^{\frac 23}\,s^{-\frac 23}, 
	 \end{equation}
	 and the monotone rearrangement inequality \cite{Lieb-Loss:2001}: 
	\begin{equation}\label{eq:I_i-monotone}
	I_i(p)= \int^{\ell_i}_0 \frac{1}{r^2(s)}ds \leq \int^{\ell_i}_0 \Bigl(\frac{1}{r^2}\Bigr)^\ast(s)ds \leq \int^{\ell_i}_0 (\tfrac{4}{3})^{\frac 23}\,s^{-\frac 23} ds= 3(\tfrac{4}{3})^{\frac 23}\,\ell_i^{\frac 13}.
	\end{equation}
	\no Integrating \eqref{eq:I_i-monotone} with respect to the $t$--parameter, we obtain 
	\[ 
	\mathsf{aov}(L)\leq \frac{1}{4\pi}\int^{\ell_j}_0\int^{\ell_i}_0 \frac{1}{||L_i(s) - L_j(t)||^2}ds\,dt\leq 3(\tfrac{4}{3})^{\frac 23}\tfrac{1}{4\pi}\, \ell_j\,\ell_i^{\frac 13}< \tfrac{1}{3} \ell_j\,\ell_i^{\frac 13}.
	\]
	Since the argument works for any choice of $i$ and $j$ estimates in Equation \eqref{eq:aov-bound} are proven. The second estimate in \eqref{eq:aov-bound} follows immediately from the fact that $\frac{1}{\|L_i(s)-L_j(t)\|^2}\leq \frac 14$.
\end{proof}
	\begin{rem}\label{rem:rearrangement}
	 {\rm	Recall, that for a nonnegative real valued function $f$, (on $\R^n$) vanishing at infinity, the {\em rearrangement, $f^\ast$ of $f$} is given  by
	 	\[
	 	 f^\ast(x)=\int^\infty_0 \chi^\ast_{\{f > u\}}(x)\, du,
	 	\]	 
	where $\chi^\ast_{\{f > u\}}(x)=\chi_{B_\rho}(x)$ is the characteristic function of the ball $B_\rho$ centered at the origin, determined by the volume condition: $\operatorname{Vol}(B_\rho)=\operatorname{Vol}(\{x\ |\ f(x)>u\})$, see \cite[p. 80]{Lieb-Loss:2001} for further properties of the rearrangements. In particular, the rearrangement inequality states \cite[p. 82]{Lieb-Loss:2001}: 
$\int_{\R^n} f(x)\,dx\leq \int_{\R^n} f^\ast(x)\,dx$. For one variable functions, we may use the interval $[0,\rho]$ in place of the ball $B_\rho$, then $f^\ast$ is a decreasing function on $[0,+\infty)$. Specifically, for $f(s)=\frac{1}{r^2(s)}=\frac{1}{\|L_i(s)-p\|^2}$, we have
\[
 \Bigl(\frac{1}{r^2}\Bigr)^\ast(s)=u,\qquad\text{for}\quad \operatorname{length}\bigl(\{x\ |\ u<\frac{1}{r^2(x)}\leq 1\}\bigr)=s,
\]  
where $\operatorname{length}\bigl(\{x\ |\ u<\frac{1}{r^2(x)}\leq 1\}\bigr)$ stands for the length of the portion of $L_i$ satisfying the given condition. Further, by the definition of $s(z)$, from the previous paragraph, and \eqref{eq:vol-tube-shell}, we obtain 
\[
s=\operatorname{length}\bigl(\{x\ |\ \frac{1}{r^2(x)}>u\}\bigr)=\operatorname{length}\bigl(\{x\ |\ 1\leq r(x)<\frac{1}{\sqrt{u}}\}\bigr)=s\bigl(\frac{1}{\sqrt{u}}\bigr)\leq \tfrac{4}{3}\, \bigl(\frac{1}{\sqrt{u}}\bigr)^3,
\]
and \eqref{eq:rng-ineq} as a result.
	}
	\end{rem}

\no From the Gauss linking integral \eqref{eq:gauss-linking-int}:
\[
 |\lk(L_i,L_j)|\leq \aov_{i,j}(L),
\]
 thus we immediately recover the result of \cite{Diao-Ernst-Janse-Van-Rensburg:2002} (but with a specific constant):
\begin{equation}\label{eq:aov-lk-bound}
3 |\lk(L_i,L_j)|\leq \min\bigl(\ell_i\ell^{\frac 13}_j,\ell_j\ell^{\frac 13}_i\bigr),\qquad 16\pi|\lk(L_i,L_j)|\leq \ell_i\ell_j.
\end{equation}
Summing up over all possible pairs: $i$, $j$ and using the symmetry of the linking number we have
\[
6\sum_{i< j} |\lk(L_i,L_j)|=3\sum_{i\neq j} |\lk(L_i,L_j)|\leq \sum_{i\neq j} \ell_i\ell^{\frac 13}_j= (\sum_i \ell_i)(\sum_j \ell^{\frac 13}_j)-\sum_i \ell^{\frac 43}_i.
\]
From Jensen's Inequality \cite{Lieb-Loss:2001}, we know that $\frac{1}{n}(\sum_i \ell^{\frac 13}_i)\leq \bigl(\frac{1}{n}\sum_i \ell_i\bigr)^{\frac 13}$ and $\frac{1}{n}(\sum_i \ell^{\frac 43}_i)\geq \bigl(\frac{1}{n}\sum_i \ell_i\bigr)^{\frac 43}$, therefore 
\begin{equation}\label{eq:lk-jensen}
(\sum_i \ell_i)(\sum_j \ell^{\frac 13}_j)-\sum_i \ell^{\frac 43}_i\leq n^{\frac 23} \rop(L)\rop(L)^{\frac 13}-n^{-\frac 13} \rop(L)^{\frac 43}=\frac{n-1}{n^\frac 13}\rop(L)^{\frac 43}.
\end{equation}
Analogously, using the second estimate in \eqref{eq:aov-lk-bound} and Jensen's Inequality, yields
\[
32\pi \sum_{i< j} |\lk(L_i,L_j)|=16\pi \sum_{i\neq j} |\lk(L_i,L_j)|\leq \sum_{i\neq j} \ell_i\ell_j\leq (1-\frac{1}{n^2})\bigl(\sum_i \ell_i\bigr)^2.
\]
\begin{cor}\label{cor:rop-lk}
 Let $L$ be an $n$-component link ($n\geq 2$), then
\begin{equation}\label{eq:lk-rop-3/4-2}
 \rop(L)^{\frac 43}\geq  \frac{6 n^{\frac{1}{3}}}{(n-1)}\sum_{i<j} |\lk(L_i,L_j)|,\qquad \rop(L)^{2}\geq \frac{32\pi n^2 }{n^2-1}\sum_{i<j} |\lk(L_i,L_j)|.
\end{equation}
\end{cor}
In terms of growth of the pairwise linking numbers $|\lk(L_i,L_j)|$, for a fixed $n$, the above estimate performs better than the one in \eqref{eq:len(L_i)-lk}. One may also replace $\sum_{i<j} |\lk(L_i,L_j)|$  with the isotopy invariant 
\begin{equation}\label{eq:PCr(L)}
\PCrn(L)=\min_{D_L} \Bigl(\sum_{i\neq j} \crn_{i,j}(D_L)\Bigr),
\end{equation}
(satisfying 
$\PCrn(L)\leq \Crn(L)$), we call the {\em pairwise crossing number} of $L$. This conclusion can be considered as an analog of the Buck and Simon estimate \eqref{eq:L(K)-BS} for knots.
%%%%%%%%%%%%%%%%%%%%%%%%%%%%%%%%%%%%%%%%%%%%
\begin{cor}\label{cor:rop-PCr}
		Let $L$ be an $n$-component link ($n\geq 2$), and $\PCrn(L)$ its pairwise crossing number, then
	\begin{equation}\label{eq2:PCr-rop-3/4-2}
 \rop(L)^{\frac{4}{3}}\geq   \frac{3 n^{\frac{1}{3}}}{(n-1)}\PCrn(L),\qquad  \rop(L)^2\geq \frac{16\pi n^2}{n^2-1}\PCrn(L).
	\end{equation}
\end{cor}

%%%%%%%%%%%%%%%%%%%%%%%%%%%%%%%%%%%%%%%%%%%%%%%%%%%%%%%%%%%%%%%
\section{Proof of Theorem \ref{thm:main}}\label{S:proof-main}

\no The following auxiliary lemma will be useful.
\begin{lem}\label{lem:a_i-ineq}
	Given nonnegative numbers: $a_1$, \ldots, $a_{N}$ we have for $k\geq 2$:
	\begin{equation}\label{eq:prod-est}
	\sum_{1\leq i_1<i_2<\ldots<i_{k}\leq N} a_{i_1}a_{i_2}\ldots a_{i_{k}}\leq \frac{1}{N^{k}} {N \choose k}\Bigl(\sum^{N}_{i=1} a_i\Bigr)^{k}.
	\end{equation}
\end{lem} 
\begin{proof}
	It suffices to observe that for $a_i\geq 0$  the ratio $\Bigl(\sum_{1\leq i_1<i_2<\ldots<i_{k}\leq N} a_{i_1}a_{i_2}\ldots a_{i_{k}}\Bigr)/\Bigl(\sum^{N}_{i=1} a_i\Bigr)^{k}$ achieves its maximum for $a_1=a_2=\ldots=a_N$. 
\end{proof}
\no Recall from \eqref{eq:[mu]} that $\bar{\mu}_{n;1}:=\bar{\mu}_{23\ldots n;1}$, and
%%%%%%%%%%%%%%%%%%%%%%%%%%%%%%%%%%%%%%%%%%%%%%%
\begin{equation}\label{eq2:[mu]}
\Big[\bar{\mu}_{n;1}(L)\Big]:=\begin{cases}
\min\Bigl(\bar{\mu}_{n;1}(L),d-\bar{\mu}_{n;1}(L)\Bigr) &\ \text{for $d> 0$},\\
|\bar{\mu}_{n;1}(L)| & \ \text{for $d= 0$}
\end{cases}\qquad d=\Delta_\mu(n;1).
\end{equation}
For convenience, recall the statement of Theorem \ref{thm:main}.
 \begin{repthm}{thm:main}
	Let $L$ be an $n$-component link of unit thickness, and $\bar{\mu}(L)$ one of its top Milnor linking numbers, then
	\begin{equation}\label{eq2:mthm-rop-cr}
	\Ln(L)\geq \sqrt[4]{n}  \Bigl(\sqrt[n-1]{\Big[\bar{\mu}(L)\Big]}\Bigr)^{\frac 34},\qquad  \Crn(L) \geq \frac{1}{3}(n-1) \sqrt[n-1]{\Big[\bar{\mu}(L)\Big]}.
	\end{equation}
\end{repthm}
\begin{proof}
Let $G_L$ be a Gauss diagram of $L$ obtained from a regular link diagram $D_L$.
Consider, any term $A$ of the arrow polynomial: $Z_{n;1}$ and index the arrows of $A$ by $(i_k,j_k)$, $k=1,\ldots,n-1$ in such a way that $i_k$ is the arrowhead and $j_k$ is the arrowtail, we have the following obvious estimate:
\begin{equation}\label{eq:estimate_term-n}
|\bigl\langle A, G_L\bigr\rangle|\leq  \prod^{n-1}_{k=1}\mathsf{ov}_{i_k,j_k}(D_L)\leq \prod^{n-1}_{k=1}\mathsf{cr}_{i_k,j_k}(D_L).
\end{equation}
Let $N={n \choose 2}$, since every term (a tree diagram) of $Z_{n;1}$ is uniquely determined by its arrows indexed by string components, ${N\choose n-1}$ gives an upper bound for the number of terms in $Z_{n;1}$. Using Lemma \ref{lem:a_i-ineq}, with $k=n-1$, $N$ as above and $a_k=\mathsf{cr}_{i_k,j_k}(D_L)$, $k=1,\ldots, N$, one obtains from \eqref{eq:estimate_term-n}
\begin{equation}\label{eq:Z_n;1-cr}
|\langle Z_{n;1},G_L\rangle|\leq \frac{1}{N^{n-1}} {N \choose n-1}\Bigl(\sum_{i<j} \mathsf{cr}_{i,j}(D_L) \Bigr)^{n-1}.
\end{equation}
\begin{rem}\label{rem:general-arrow-poly}
	{\em 
The estimate \eqref{eq:Z_n;1-cr} is valid for any arrow polynomial, in place of $Z_{n;1}$, which has arrows based on different components and no parallel arrows on a given component.}
\end{rem}
By \eqref{eq:bar-mu-Z_n;1}, we can find $k\in \Z$ such that $\langle Z_{n;1},G_L\rangle=\bar{\mu}_{n;1}+k\,d$. Since
\[
 \Big[\bar{\mu}_{n;1}(D_L)\Big]\leq |\bar{\mu}_{n;1}(D_L)+k\,d|=|\langle Z_{n;1},G_L\rangle|,\qquad\text{for all}\ k\in\Z,
\]
replacing $D_L$ with a diagram obtained by projection of $L$ in a generic direction $v\in S^2$, we rewrite the estimate \eqref{eq:Z_n;1-cr} as follows
\begin{equation}\label{eq:alpha_n-cr}
\alpha_n\sqrt[n-1]{\Big[\bar{\mu}_{n;1}(D_L(v))\Big]}\leq \sum_{i<j} \mathsf{cr}_{i,j}(v),\qquad \alpha_n=\Bigl(\frac{1}{ N^{n-1}} {N \choose n-1}\Bigr)^{\frac{-1}{n-1}}.
\end{equation}
Integrating over the sphere of directions and using invariance\footnote{both $\bar{\mu}_{n;1}$ and $d$ are isotopy invariants.} of $\Big[\bar{\mu}_{n;1}\Big]$ yields
\[
4\pi \alpha_n \sqrt[n-1]{\Big[\bar{\mu}_{n;1}(L)\Big]}\leq \sum_{i<j} \int_{S^2}\mathsf{cr}_{i,j}(v) dv.
\]
By Lemma \ref{lem:aov-bound}, we obtain
\[
\alpha_n \sqrt[n-1]{\Big[\bar{\mu}_{n;1}(L)\Big]}\leq \sum_{i<j} \mathsf{acr}_{i,j}(L)= 2\sum_{i<j} \mathsf{aov}_{i,j}(L)\leq 2\sum_{i< j} \tfrac{1}{3}\min\bigl(\ell_i\ell^{\frac 13}_j,\ell_j\ell^{\frac 13}_i\bigr)\leq \tfrac{1}{3}\sum_{i \neq j} \ell_i\ell^{\frac 13}_j,
\]
since $\sum_{i< j} 2\min\bigl(\ell_i\ell^{\frac 13}_j,\ell_j\ell^{\frac 13}_i\bigr)\leq \sum_{i \neq j} \ell_i\ell^{\frac 13}_j$.
As in derivation of \eqref{eq:lk-jensen} (see Corollary \ref{cor:rop-lk}), by Jensen Inequality: 
\begin{equation}\label{eq:rop-end}
\rop(L)^{\frac 43}\geq \frac{3\, n^{\frac 13}\, \alpha_n}{(n-1)}  \sqrt[n-1]{\Big[\bar{\mu}_{n;1}(L)\Big]}.
\end{equation}
Now, let us estimate the constant $\alpha_n$. Note that 
\[
\frac{N^{n-1}}{{N \choose n-1}}=\frac{N^{n-1}}{N(N-1)\ldots (N-(n-1)+1)} (n-1)!\geq (n-1)!\ .
\]
Again, by Stirling's approximation  (letting $m=n-1$ in \eqref{eq:stirling2}) we obtain for $n\geq 2$:
\begin{equation}\label{eq:alpha_n-stirling}
\alpha_n\geq \bigl((n-1)!\bigr)^{\frac{1}{n-1}}\geq \frac{n-1}{3},
\end{equation}
thus \eqref{eq:rop-end} can be simplified to
\begin{equation}\label{eq:rop-end2}
\rop(L)^{\frac 43}\geq \sqrt[3]{n}  \sqrt[n-1]{\Big[\bar{\mu}_{n;1}(L)\Big]},
\end{equation}
as claimed in the first inequality of Equation \eqref{eq2:mthm-rop-cr}. For a minimal diagram $D^{\min}_L$ of $L$: 
\[
\Crn(L)\geq \sum_{i<j} \mathsf{cr}_{i,j}(D^{\min}_L),
\]
 thus the second inequality of \eqref{eq2:mthm-rop-cr} is an immediate consequence of \eqref{eq:alpha_n-cr}(with $D_L(v)$ replaced by $D^{\min}_L$) and \eqref{eq:alpha_n-stirling}.
Using the permutation identity \eqref{eq:mu-symmetry} and the fact that $\rop(\sigma(L))=\rop(L)$ for any $\sigma\in\Sigma(1,\ldots,n)$, we may replace $\bar{\mu}_{n;1}(L)$ with any other\footnote{there are $(n-2)!$ different top Milnor linking numbers \cite{Milnor:1954}.} top $\bar{\mu}$--invariant of $L$.
\end{proof}
	In the case of almost trivial (Borromean) links $d=0$, and we may slightly improve the estimate in \eqref{eq:Z_n;1-cr} of the above proof, by using cyclic symmetry of $\bar{\mu}$--invariants pointed in \eqref{rem:mu-cyclic}. We have in particular
	\begin{equation}\label{eq:mu-sum-cyclic}
	n\,\bar{\mu}_{23\ldots n;1}(L)=\sum_{\rho, \text{$\rho$ is cyclic}} \bar{\mu}_{\rho(2)\rho(3)\ldots\rho(n);\rho(1)}(L)=\sum_{\rho, \text{$\rho$ is cyclic}} \langle \rho(Z_{n;1}),G_L\rangle.
	\end{equation}
	Since cyclic permutations applied to the terms of $Z_{n;1}$ produce distinct arrow diagrams\footnote{since the trunk of a tree diagram is unique, c.f. \cite{Kravchenko-Polyak:2011}, \cite{Komendarczyk-Michaelides:2016}.}, by Remark \ref{rem:general-arrow-poly}, we obtain the following bound 
	\begin{equation}\label{eq:rho-Z_n;1-cr}
	n\,|\bar{\mu}_{n;1}(L)|\leq \sum_{\rho, \text{$\rho$ is cyclic}} |\langle \rho(Z_{n;1}),G_L\rangle|\leq \frac{1}{N^{n-1}} {N \choose n-1}\Bigl(\sum_{i<j} \mathsf{cr}_{i,j}(D_L) \Bigr)^{n-1}.
	\end{equation}
     Disregarding the Stirling's approximation, we have
     \begin{equation}
     \rop(L)^{\frac 43}\geq \frac{3\sqrt[3]{n}\, \tilde{\alpha}_n}{(n-1)}  \sqrt[n-1]{|\bar{\mu}_{n;1}(L)|},\quad \tilde{\alpha}_n=\Bigl(\frac{1}{n N^{n-1}} {N \choose n-1}\Bigr)^{\frac{-1}{n-1}},
     \end{equation}
     or using the second bound in \eqref{eq:aov-bound}
     \[
     \rop(L)^2\geq 4^3\pi \tilde{\alpha}_n\Bigl(\frac{n^2}{n^2-1}\Bigr)  \sqrt[n-1]{|\bar{\mu}_{n;1}(L)|}.
     \]	
     In particular, for $n=3$, we have $N=3$ and $\tilde{\alpha}_3=3$ and the estimates read
\begin{equation}\label{eq:rop-mu123}
	\rop(L)\geq \Bigl(5\sqrt[3]{3}\sqrt{|\bar{\mu}_{23;1}(L)|}\Bigr)^{\frac 34},\quad \rop(L)\geq 6\sqrt{6\pi}\sqrt[4]{|\bar{\mu}_{23;1}(L)|}.
\end{equation}

\no Since $6\sqrt{6\pi}\approx 26.049$, the second estimate is better for Borromean rings ($\mu_{23;1}=1$) and improves the linking number bound of \eqref{eq:len(L_i)-lk}: $6\pi\approx 18.85$, but fails short of the genus bound \eqref{eq:len(L_i)-Ac-gen}: $12\pi\approx 37.7$. Numerical simulations  suggest that the ropelength of Borromean rings $\approx 58.05$, \cite{Cantarella-Kusner-Sullivan:2002, Buniy-Cantarella-Kephart-Rawdon:2014}.

\begin{rem}\label{rem:method-extended}
{\em
	This methodology can be easily extended to other families of finite type invariants of knots and links.  For illustration, let us consider the third coefficient of the Conway polynomial i.e. $c_3(L)$ of a two component link $L$. The arrow polynomial $C_3$ of $c_3(L)$ is given as follows \cite[p. 779]{Chmutov-Khoury-Rossi:2009}
	\[
	\vvcenteredinclude{.31}{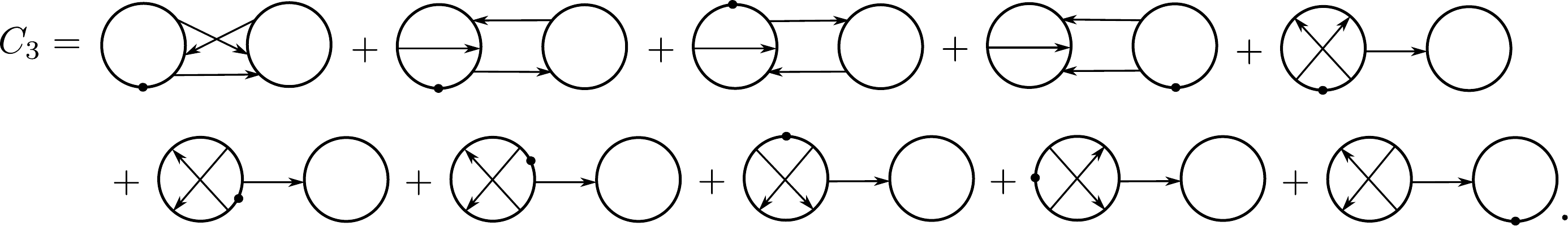}
	\]
    Let $G_L$ be the Gauss diagram obtained from a regular link diagram $D_L$, and $D_{L_k}$ the subdiagram of the $k$th--component of $L$, $k=1,2$. The absolute value of the first term $\langle \vvcenteredinclude{.25}{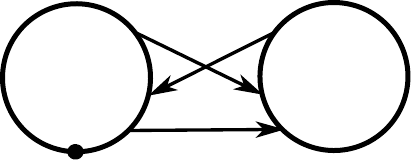}, G_L\rangle$ of $\langle C_3, G_L\rangle$ does not exceed ${\crn_{1,2}(D_L) \choose 3}$, the absolute value of the sum $\langle \vvcenteredinclude{.25}{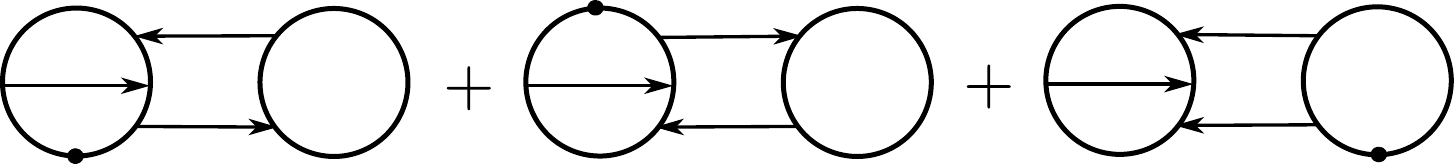}, G_L\rangle$ does not exceed $\crn(D_{L_1}){\crn_{1,2}(D_{L}) \choose 2}$ and for the remaining terms a bound is ${\crn(D_{L_1}) \choose 2}\crn_{1,2}(D_{L})$. Therefore, a rough upper bound for $|\langle C_3, G_L\rangle|$ can be written as 
    \[
      |\langle C_3, G_L\rangle|\leq (\crn_{1,2}(D_{L})+\crn(D_{L_1}))^3 %\crn_{1,2}(D_{L})^3+\crn(D_{L_1}) \crn_{1,2}(D_{L})^2+\crn(D_{L_1})^2 \crn_{1,2}(D_{L}).
     \] 
       
\no   Similarly, as in \eqref{eq:alpha_n-cr}, replacing $D_L$ with $D_L(v)$ and integrating over the sphere of directions we obtain
\[
 |c_3(L)|^{\frac 13}\leq \mathsf{acr}_{1,2}(L)+\mathsf{acr}(L_1).   
\]
For a unit thickness link $L$, \eqref{eq:rop-BS-acr} and \eqref{eq:aov-bound} give 
$\mathsf{acr}_{1,2}(L)+\mathsf{acr}(L_1)\leq \frac{1}{3}\ell^{\frac 13}_1\ell_2+\frac{1}{3}\ell^{\frac 13}_2\ell_1+\frac{4}{11}\ell^{\frac 13}_1\ell_1$, and $\mathsf{acr}_{1,2}(L)+\mathsf{acr}(L_1)\leq \frac{\ell^2_1}{16\pi}+\frac{\ell_1\ell_2}{8\pi}$. Thus, for some constants: $\alpha$, $\beta>0$, we have
\[
 \ell(L)^2\geq A\, |c_3(L)|^{\frac 13},\qquad \ell(L)^{\frac{4}{3}}\geq B\, |c_3(L)|^{\frac 13}.
\]
\no In general, given a finite type $n$ invariant $V_n(L)$ and a unit thickness  $m$--link $L$, we may expect constants $\alpha_{m,n}$, $\beta_{m,n}$; such that
\[
 \ell(L)^2\geq \alpha_{m,n}\, |V_n(L)|^{\frac{1}{n}},\qquad \ell(L)^{\frac{4}{3}}\geq \beta_{m,n}\, |V_n(L)|^{\frac{1}{n}}.
\]
}
\end{rem}
%%%%%%%%%%%%%%%%%%%%%%%%%%%%%%%%%%%

\end{document}